\documentclass[12pt]{article}

\usepackage{amsmath}
\usepackage{amssymb}
\usepackage{amsthm}
\usepackage[noadjust]{cite}
\usepackage{mathrsfs}
\usepackage{dsfont}
\usepackage{dcpic}
\usepackage{hyperref}
\usepackage[all,cmtip]{xy}
\usepackage{xcolor}
\usepackage{enumerate}
\usepackage{bbm}
\usepackage{dsfont}
\usepackage{verbatim}
\usepackage{mathtools}
\usepackage{bm}

\DeclarePairedDelimiter\floor{\lfloor}{\rfloor}

\hypersetup{colorlinks=true,linkcolor=red,citecolor=blue}

\newtheorem{lemma}{Lemma}[section]
\newtheorem{proposition}[lemma]{Proposition}
\newtheorem{theorem}[lemma]{Theorem}
\newtheorem{corollary}[lemma]{Corollary}
\newtheorem{conjecture}[lemma]{Conjecture}

\theoremstyle{definition}
\newtheorem{definition}[lemma]{Definition}
\newtheorem{remark}[lemma]{Remark}

\def\C{{\mathbb C}}

\def\Gg{\mathcal{G}}
\def\Hh{\mathcal{H}}

\def\Ll{\mathcal{L}}
\def\Mm{\mathcal{M}}
\def\Nn{\mathcal{N}}

\def\Pp{\mathcal{P}}

\def\Rr{\mathcal{R}}

\def\AA{\mathfrak{A}}
\def\BB{\mathfrak{B}}

\def\II{\mathfrak{I}}

\def\MM{\mathfrak{M}}

\def\PP{\mathfrak{P}}

\def\SS{\mathfrak{S}}

\def\III{\mathbf{I}}

\def\NNN{\mathbf{N}}

\def\XXX{\mathbf{X}}

\def\ker{\operatorname{ker}}
\def\Hom{\operatorname{Hom}}
\def\End{\operatorname{End}}

\def\Ind{\operatorname{Ind}}
\def\cInd{c\text{-}\operatorname{Ind}}
\def\Res{\operatorname{Res}}
\def\Gal{\operatorname{Gal}}
\def\Irr{\mathrm{Irr}}
\def\Mat{\mathrm{Mat}}
\def\GL2{\mathbf{GL}_2}
\def\SL2{\mathbf{SL}_2}
\def\GLN{\mathbf{GL}_N}
\def\SLN{\mathbf{SL}_N}
\def\PGL{\mathbf{PGL}}
\def\Rep{\mathrm{Rep}}
\def\rec{\mathrm{rec}}
\newcommand{\dedekind}{{\scriptstyle\mathcal{O}}}
\def\XXXnr{\mathbf{X}_{\mathrm{nr}}}
\def\St{\mathrm{St}}

\makeatletter
\renewcommand\subsection{\@startsection {subsection}{1}{\z@}%
                                        {-3.5ex \@plus -1ex \@minus -.2ex}%
                                        {2.3ex \@plus.2ex}%
                                        {\normalfont\textit}}
\makeatother
\begin{document}
\newtheorem*{recall}{Recall}
\newtheorem*{notat}{Notations}
\newtheorem*{problem}{Problem}
\newtheorem*{fact}{Fact}

\title{Unicity of types for supercuspidal representations of $p$-adic $\SL2$}
\author{Peter Latham\footnote{The University of East Anglia, email: peter.latham@uea.ac.uk}}
\date{\vspace{-5ex}}
\maketitle

\abstract{\noindent\normalsize We consider the question of unicity of types on maximal compact subgroups for supercuspidal representations of $\SL2$ over a nonarchimedean local field of odd residual characteristic. We introduce the notion of an archetype as the $\SL2$-conjugacy class of a typical representation of a maximal compact subgroup, and go on to show that any archetype in $\SL2$ is restricted from one in $\GL2$. From this it follows that any archetype must be induced from a Bushnell--Kutzko type. Given a supercuspidal representation $\pi$, we give an additional explicit description of the number of archetypes admitted by $\pi$ in terms of its ramification. We also describe a relationship between archetypes for $\GL2$ and $\SL2$ in terms of $L$-packets, and deduce an inertial Langlands correspondence for $\SL2$.}
\normalsize
\setlength{\parindent}{0pt}

\section{Introduction}

The local Langlands conjectures provide natural correspondences between certain representations of the Weil--Deligne group of a nonarchimedean local field $F$ and (packets of) smooth, irreducible representations of certain reductive groups defined over $F$. While this is theoretically very important, it is rather difficult to obtain explicit information from this correspondence. The theory of types arose as a means of obtaining an explicit understanding of the representation theory of $p$-adic groups, partly in the hope that this would allow one to obtain explicit information on the Weil group side of the Langlands correspondence.\\

Following Bernstein, one may factorize the category of smooth representations of a reductive $p$-adic group $\Gg$ into a direct product of full subcategories consisting of representations whose irreducible subquotients, viewed as representations of the Weil--Deligne group via local Langlands, identify upon restriction to the inertia group. This suggests an approach to obtaining explicit information about the representation theory of $\Gg$ by restriction of these natural families of representations to the compact open subgroups of $\Gg$. Formally, given a Bernstein component $\Rr$ of the category of smooth representations of $\Gg$, one says that a \emph{type} for $\Rr$ is a smooth irreducible representation $\lambda$ of a compact open subgroup $J$ of $\Gg$, such that an irreducible representation $\pi$ of $\Gg$ contains $\lambda$ upon restriction to $J$ if and only if $\pi\in\Rr$. Similarly, we say that $(J,\lambda)$ is \emph{typical} for $\Rr$ if any irreducible representation $\pi$ containing $\lambda$ must be contained in $\Rr$. Thus, types provide a means of transferring problems regarding the (infinite-dimensional) representation theory of $\Gg$ into problems regarding the (finite-dimensional) representation theory of its compact open subgroups, allowing for more explicit results. One would then hope to be able to transfer this information to the Weil group side of the Langlands correspondence in order to obtain results about Weil--Deligne representations in terms of their restriction to inertia. Results in this direction have been highly influential, allowing for the proof of an inertial local Langlands correspondence for $\GLN(F)$ in \cite{paskunas2005unicity}, and for progress towards the Breuil--M\'{e}zard conjecture, as in \cite{breuil2002multiplicites}.\\

As with the representation theory of $\Gg$, the natural approach to the construction of types is to proceed by parabolic induction -- that is, to construct types for the supercuspidal representations of all Levi subgroups of $\Gg$, and then to find some operation compatible with parabolic induction which will allow for the construction of types for all of the Bernstein components of $\Gg$ (which we now know to be given by the Bushnell--Kutzko theory of covers, as in \cite{bushnell1998structure}). Considerable progress has been made in this direction: there are now constructions of types for all of the Bernstein components of $\Gg$ when $\Gg$ is a general or special linear group (\cite{bushnell1993admissible}, \cite{bushnell1993sln}, \cite{bushnell1994sln}, \cite{bushnell1999semisimple} and \cite{goldberg2002types}), when $\Gg$ is a classical group in odd residual characteristic (\cite{stevens2008supercuspidals} and \cite{Miyauchi2014semisimple}), and when $\Gg$ is an inner form of a general linear group (\cite{secherre2005simpletypes} \cite{secherre2008supercuspidals} and \cite{secherre2012semisimple}), and constructions of types for all tamely ramified supercuspidal representations of arbitrary groups $\Gg$ over fields of characteristic zero (\cite{yu2001construction} and \cite{kim2007exhaustion}).\\

In each of these cases, the approach has been to construct, by a series of successively stronger approximations, an explicit type for each Bernstein component of $\Gg$. Moreover, all constructions of types for supercuspidals to date have led to resolutions of the long-standing folklore conjecture that any supercuspidal representation of $\Gg$ should be compactly induced from an irreducible representation of an open, compact-modulo-centre subgroup of $G$. However, it is \emph{a priori} unclear from the abstract definition of a type that there should not exist other examples of types, which are perhaps less suited to applications. This naturally raises the question of the ``unicity of types'', which asks whether or not this is the case.\\

Call any type arising from the constructions discussed above a ``Bushnell--Kutzko  type''. It is simple to see that, given a Bushnell--Kutzko type $(J,\lambda)$ and a maximal compact subgroup $K$ of $\Gg$ which contains $J$, the irreducible components of the representation obtained by inducing $\lambda$ to $K$ are typical representations. One then wishes to show that repeating this process for all Bushnell--Kutzko types provides a complete list of typical representations of maximal compact subgroups of $\Gg$. As well as being of interest in itself as a resolution of the theory of types for $\Gg$, such a result also turns out to be precisely what is required in order to allow applications to results such as inertial Langlands correspondences and Breuil--M\'{e}zard-like conjectures.\\

The question of unicity was first considered by Henniart in the appendix to \cite{breuil2002multiplicites}, precisely in order to allow an application to the Breuil--M\'{e}zard conjecture, where he obtains a positive answer for all representations of $\GL2(F)$ (at least when the cardinality of the residue field is at least $3$). Since then, the result has been extended to cover all supercuspidal representations of $\GLN(F)$ by Paskunas in \cite{paskunas2005unicity} and to cover all representations of $\mathbf{GL}_3(F)$ (\cite{nadimpalli2014gl3}), and most representations of $\mathbf{GL}_4(F)$, as well as all representations of $\GLN(F)$ of depth zero by Nadimpalli in work to appear.\\

In this paper, we take the first steps towards unicity for special linear groups, providing a positive answer for supercuspidal representations of $\SL2(F)$ when $F$ is of odd residual characteristic. The main difference between the cases of $\GL2(F)$ and $\SL2(F)$ is that there are now two conjugacy classes of maximal compact subgroups. Given an irreducible representation $\pi$ of $\SL2(F)$, there will clearly always be a typical representation of at least one of these maximal compact subgroups, obtained by inducing up a Bushnell--Kutzko type for $\pi$, but not necessarily on both. In order to deal with these complications, we introduce the notion of an \emph{archetype}, which is an $\SL2(F)$-conjugacy class of typical representations $(\mathscr{K},\tau)$ for some maximal compact subgroup $\mathscr{K}$. With this in place, we are able to give a natural extension of the unicity of types to this setting, providing a positive answer for supercuspidal representations in section \ref{unicity}.\\

In section \ref{supercuspidals}, we go on to provide a more explicit description of the archetypes for the supercuspidal representations of $\SL2(F)$, allowing the following refinement of our unicity result:
\begin{theorem}
Let $\pi$ be a supercuspidal representation of $\SL2(F)$, for $F$ a nonarchimedean local field of odd residual characteristic. If $\pi$ is of integral depth, then there exists a unique archetype for $\pi$, while if $\pi$ is of half-integral depth then there exist precisely two archetypes for $\pi$, which are $\GL2(F)$-conjugate but not $\SL2(F)$-conjugate.
\end{theorem}
We also provide, in Proposition \ref{lpackets}, a description of the relationship between supercuspidal archetypes in $\GL2(F)$ and $\SL2(F)$ in terms of the local Langlands correspondence, which in some sense says that archetypes are functorial with respect to restriction from $\GL2(F)$ to $\SL2(F)$. This allows us to deduce in Corollary \ref{inertialcorrespondence} an extension of Paskunas' inertial Langlands correspondence to our setting.\\

Our method is to transfer Henniart's results on $\GL2(F)$ over to $\SL2(F)$, with the key step being to show that any archetype for an irreducible representation $\bar{\pi}$ of $\SL2(F)$ must be isomorphic to an irreducible component of the restriction of the unique archetype for some irreducible representation $\pi$ of $\GL2(F)$ containing $\bar{\pi}$ upon restriction. This is achieved in Lemma \ref{extensiontok}. From this, it is mostly a case of performing simple calculations to deduce in Theorem \ref{typesaresimple} that our unicity result holds. The explicit counting result on the number of archetypes contained in a supercuspidal representation follows easily from Theorem \ref{typesaresimple}, while we are able to prove in Lemma \ref{restrictiontokbar} a form of converse to Lemma \ref{extensiontok} for the supercuspidal representations, which allows us to easily deduce the remaining results.\\

While we have avoided doing so in this paper, one could have proved the same results by essentially copying the methods used by Henniart for $\GL2(F)$. One may show unicity with respect to a fixed choice of maximal compact subgroup by following Henniart's approach, making only the necessary changes, with the only additional complication being the proof that the integral depth supercuspidal representations admit only a single archetype. For the positive depth representations, this is achievable using a minor variation of Henniart's arguments, but the depth zero representations require more work. The author knows of two approaches in this case: to use the branching rules found in \cite{nevins2013supercuspidal}, or to argue using covers (in the sense of \cite{bushnell1998structure}). The problem with this approach is that there is necessarily a large amount of duplication of effort. While one would expect that such an approach could be made to work for arbitrary $N$, this would require reproving most of the results found in \cite{paskunas2005unicity} with only minor modifications.\\

On the other hand, the approach taken in this paper is largely general, and already gives partial progress towards a general proof of the unicity of types for $\SLN(F)$. In particular, the proof of Lemma \ref{order2} goes through in the general setting without any additional difficulties, suggesting the possibility of applying the results of \cite{paskunas2005unicity} in a similar manner to our use of Henniart's arguments in order to prove an analogue of Lemma \ref{extensiontok}, which the author is hopeful of managing in the near future. In particular, this would lead easily to a positive answer to the question of unicity. The remaining results should then follow without too much difficulty in a similar manner to that here. In particular, this should allow for the following extension of our explicit results on supercuspidals. Given a supercuspidal representation $\bar{\pi}$ of $\SLN(F)$ and a supercuspidal representation $\pi$ of $\GLN(F)$ which contains $\bar{\pi}$ upon restriction, we define the \emph{ramification degree} of $\bar{\pi}$ to be the number $e_{\bar{\pi}}$ such that there are $N/e_{\bar{\pi}}$ characters $\chi$ of $F^\times$ such that $\pi\simeq\pi\otimes(\chi\circ\det)$. This is independent of the choice of $\pi$. Then we make the following conjecture:
\begin{conjecture}
Let $\bar{\pi}$ be a supercuspidal representation of $\SLN(F)$. Then there are precisely $e_{\bar{\pi}}$ archetypes for $\bar{\pi}$, which are $\GLN(F)$-conjugate.
\end{conjecture}

\subsection{Acknowledgements}

This paper has resulted from my first year of PhD work at the University of East Anglia. It has benefited hugely from discussions with my advisor, Shaun Stevens, who is due thanks for his constant encouragement and generosity with ideas. I am also grateful to the EPSRC for providing my research studentship.

\subsection{Notation}

Throughout, $F$ will denote a nonarchimedean local field of odd residual characteristic $p$. We will denote by $\dedekind=\dedekind_F$ the ring of integers of $F$, and write $\mathfrak{p}=\mathfrak{p}_F$ for its maximal ideal. The residue field will be denoted by $\mathfrak{k}=\mathfrak{k}_F=\dedekind/\mathfrak{p}$, and we will write $q$ for the cardinality of $\mathfrak{k}$. We fix once and for all a choice $\varpi$ of uniformizer of $F$, i.e. an element such that $\varpi\dedekind=\mathfrak{p}$.\\

When working in generality, we will use $\Gg$ to denote an arbitrary $p$-adic group defined over $F$, by which we will mean the group $\Gg=\bm{\Gg}(F)$ of $F$-rational points of some connected reductive algebraic group $\bm{\Gg}$ defined over $F$. We will always denote by $G$ the general linear group $\GL2(F)$. We fix notation for a number of important subgroups of $G$. We will write $K=\GL2(\dedekind)$ for the standard maximal compact subgroup, $T$ for the split maximal torus of diagonal matrices, and $B$ for the standard Borel subgroup of upper triangular matrices. We also write $\bar{G}$ for the special linear group $\SL2(F)$ and, given a closed subgroup $H$ of $G$, we let $\bar{H}$ denote the subgroup $H\cap\bar{G}$ of $\bar{G}$. We also denote by $T^0$ the compact part of the torus, i.e. the group of diagonal matrices with entries in $\dedekind^\times$, and by $B^0=B\cap K$ the group of upper triangular matrices with entries in $\dedekind$. We will denote by $\eta$ the matrix $\begin{pmatrix}0 & 1\\ \varpi & 0\end{pmatrix}$, so that we may take $\bar{K}$ and $\eta\bar{K}\eta^{-1}$ as representatives of the two $\bar{G}$-conjugacy classes of maximal compact subgroups in $\bar{G}$.\\

We use the notation $^gx=gxg^{-1}$ for conjugation, similarly denoting by $^gX=\{{}^gx\ | \ x\in X\}$ the action of conjugation on a set. Given a representation $\sigma$ of a closed subgroup $\Hh$ of $\Gg$, we denote by $^g\sigma$ the representation of $^g\Hh$ given by $^g\sigma(ghg^{-1})=\sigma(h)$.\\

We write $\Rep(\Gg)$ for the category of smooth representations of $\Gg$, and $\Irr(\Gg)$ for the set of isomorphism classes of irreducible representations in $\Rep(\Gg)$. Given a closed subgroup $\Hh$ of $\Gg$, we write $\Ind_{\Hh}^{\Gg}\ \sigma$ for the smooth induction of $\sigma$ to $\Gg$, and $\cInd_{\Hh}^{\Gg}\ \sigma$ for the compact induction. We write $\Res_{\Hh}^{\Gg}\ \pi$ for the restriction of $\pi$ to $\Hh$, or simply $\pi\downharpoonright_{\Hh}$ for brevity when it is unnecessary to make clear the functor. Given subgroups $\Hh,\Hh'$ of $\Gg$ and representations $\lambda,\lambda'$ of $\Hh,\Hh'$, respectively, we write $\III_{\Gg}(\lambda,\lambda')=\{g\in\Gg\ | \ \Hom_{\Hh\cap{}^g\Hh'}(\lambda,{}^g\lambda')\neq 0\}$ for the intertwining of $\lambda$ with $\lambda'$.\\

Given a parabolic subgroup $\Pp$ of $\Gg$ with Levi decomposition $\Pp=\Mm\Nn$, we denote the \emph{normalized} parabolic induction of an irreducible representation $\zeta$ of $\Mm$ to $\Gg$ by $\Ind_{\Mm,\Pp}^{\Gg}\ \zeta$. By this, we mean $\Ind_{\Mm,\Pp}^{\Gg}\ \zeta=\Ind_\Pp^\Gg\ \tilde{\zeta}\otimes\delta_P^{-1/2}$, where $\tilde{\zeta}$ is the inflation of $\zeta$ to $\Pp$ and $\delta_P$ is the modular character of $\Pp$.\\

Finally, we denote by $\XXX(F)$ the group of complex characters $\chi:F^\times\rightarrow\C^\times$. We will be interested in two subgroups of this: the group $\XXXnr(F)$ of unramified characters in $\XXX(F)$ (i.e. those which are trivial on $\dedekind^\times$), and the group $\XXX_N(F)$ of order $N$ characters in $\XXX(F)$ (i.e. those $\chi\in\XXX(F)$ such that $\chi^N=\mathds{1}$).

\subsection{The Bernstein decomposition and types}

The Bernstein decomposition, which was first introduced in \cite{bernstein1984centre}, allows us to give a factorization of the category $\Rep(\Gg)$, which suggests a natural approach to its study. Given an irreducible representation $\pi$ of $\Gg$, there exists a unique $\Gg$-conjugacy class of smooth irreducible representations $\sigma$ of Levi subgroups $\Mm$ of $\Gg$ such that $\pi$ is isomorphic to an irreducible subrepresentation of $\Ind_{\Mm,\Pp}^\Gg\ \sigma$, for some parabolic subgroup $\Pp$ of $\Gg$ with Levi factor $\Mm$. We call this equivalence class the \emph{supercuspidal support} of $\pi$, and denote it by $\mathrm{scusp}(\pi)$. We put a further equivalence relation on the set of possible supercuspidal supports, by saying that $(\Mm,\sigma)$ is $\Gg$-\emph{inertially equivalent} to $(\Mm',\sigma')$ if there exists a $\chi\in\XXXnr(F)$ such that $(\Mm,\sigma)$ is $\Gg$-conjugate to $(\Mm',\sigma'\otimes\chi)$. The \emph{inertial support} of $\pi$ is then the inertial equivalence class of $\mathrm{scusp}(\pi)$. If $\mathrm{scusp}(\pi)=(\Mm,\sigma)$, then we write $[\Mm,\sigma]_{\Gg}$ for the inertial support of $\pi$.\\

With this in place, let $\BB(\Gg)$ denote the set of inertial equivalence classes of supercuspidal supports, and, for $\mathfrak{s}\in\BB(\Gg)$, let $\Rep^{\mathfrak{s}}(\Gg)$ denote the full subcategory of $\Rep(\Gg)$ consisting of representations such that all irreducible subquotients have inertial support $\mathfrak{s}$, and write $\Irr^{\mathfrak{s}}(\Gg)$ for the set of isomorphism classes of irreducible representations in $\Rep^{\mathfrak{s}}(\Gg)$. Bernstein then shows that
\[\Rep(\Gg)=\prod_{\mathfrak{s}\in\BB(\Gg)}\Rep^{\mathfrak{s}}(\Gg).
\]
More generally, given a subset $\SS$ of $\BB(\Gg)$, let $\Rep^\SS(\Gg)=\prod_{\mathfrak{s}\in\SS}\Rep^{\mathfrak{s}}(\Gg)$ and $\Irr^\SS(\Gg)=\bigcup_{\mathfrak{s}\in\SS}\Irr^{\mathfrak{s}}(\Gg)$. This allows us to define the notion of a type in generality:

\begin{definition}
Let $\SS\subset\BB(\Gg)$. Let $(J,\lambda)$ be a pair consisting of a compact open subgroup $J$ of $\Gg$ and a smooth irreducible representation $\lambda$ of $J$.
\begin{enumerate}[(i)]
\item We say that $(J,\lambda)$ is $\SS$-\emph{typical} if, for any smooth irreducible representation $\pi$ of $\Gg$, we have that $\Hom_J(\pi\downharpoonright_J,\lambda)\neq 0\Rightarrow\pi\in\Irr^{\SS}(\Gg)$.
\item We say that $(J,\lambda)$ is an $\SS$-\emph{type} if it is $\SS$-typical, and $\Hom_J(\pi\downharpoonright_J,\lambda)\neq 0$ for each $\pi\in\Irr^{\SS}(\Gg)$.
\end{enumerate}
In the case that $\SS=\{\mathfrak{s}\}$ is a singleton, we will simply speak of $\mathfrak{s}$-types rather than $\{\mathfrak{s}\}$-types.
\end{definition}

In the cases of interest to us, the Bernstein components of $\Rep(\Gg)$ admit particularly simple descriptions: if $\Rep^{\frak{s}}(G)$ contains a supercuspidal representation $\pi$, then $\Irr^{\frak{s}}(G)=\{\pi\otimes(\chi\circ\det)\ | \ \chi\in\XXXnr(F)\}$. The situation for $\bar{G}$ is even simpler: as $\bar{G}$ has no unramified characters, $\Irr^{\frak{s}}(\bar{G})$ is a singleton whenever it contains a supercuspidal representation.\\

We now introduce the slightly modified notion of an \emph{archetype}, which is more suited to studying the unicity of types in groups other than $\GLN(F)$.

\begin{definition}
Let $\SS\subset\BB(\Gg)$. An $\SS$-\emph{archetype} is a $\Gg$-conjugacy class of $\SS$-typical representations $(\mathscr{K},\tau)$ for $\mathscr{K}$ a maximal compact subgroup of $\Gg$. Given a representative $(\mathscr{K},\tau)$ of an archetype, we write $^{\Gg}(\mathscr{K},\tau)$ for the full conjugacy class.
\end{definition}

\begin{remark}
It may seem odd to define an archetype as a conjugacy class of \emph{typical} representations rather than as a conjugacy class of types. However, for us, the difference turns out to be unimportant: the unicity of types will allow us to see that typical representations of maximal compact subgroups are types in almost all cases (indeed, for all representations not contained in the restriction of the Steinberg representation of $G$). The reason for working with typical representations rather than types is that it allows us to include these ``Steinberg'' representations in the general picture, despite them admitting no type of the form $(\mathscr{K},\tau)$.
\end{remark}

There is one obvious way of constructing archetypes:

\begin{lemma}
Let $\pi$ be an irreducible representation of a $p$-adic group $\Gg$ of inertial support $\mathfrak{s}$. Let $(J,\lambda)$ be an $\mathfrak{s}$-type, and let $\mathscr{K}$ be a maximal compact subgroup of $\Gg$ containing $J$. Then the irreducible components of $\tau:=\cInd_J^{\mathscr{K}}\ \lambda$ are representatives of $\mathfrak{s}$-archetypes. Moreover, if $\tau$ is irreducible then it is an $\mathfrak{s}$-type.
\end{lemma}

\begin{proof}
Using Frobenius reciprocity, it is clear that if an irreducible representation $\pi'$ of $\Gg$ contains $\tau$, then it must contain $\lambda$, hence the first claim. The second claim simply follows by the transitivity of induction.
\end{proof}

The question of the unicity of types is then whether there are any archetypes other than those induced from Bushnell--Kutzko types. For $G$, Henniart answers this in the appendix to \cite{breuil2002multiplicites}:

\begin{theorem}\label{unicitygl2}
Suppose $q\neq 2$. Let $\pi$ be an irreducible representation of $G$ of inertial support $\mathfrak{s}$. Let $^G(K,\tau)$ be an $\frak{s}$-archetype. Then there exists a Bushnell--Kutzko type $(J,\lambda)$ with $J\subset K$ such that $\tau\hookrightarrow\cInd_J^K\ \lambda$. Moreover, unless $\pi$ is a twist of the Steinberg representation $\St_G$, the representation $\cInd_J^K\ \lambda$ is irreducible and hence $(K,\tau)$ is an $\frak{s}$-type.
\end{theorem}

\subsection{Bushnell--Kutzko types}

We now describe the explicit construction, due to Bushnell and Kutzko, of types for the irreducible representations of $G$ and $\bar{G}$. In this section, we discuss the types for supercuspidal representations, which are the \emph{simple types} constructed in \cite{bushnell1993admissible}, \cite{bushnell1993sln} and \cite{bushnell1994sln}. The construction of these types is by a series of successively stronger approximations of a type, and is rather technical in nature. We omit as many details as possible; the full details for our case of $N=2$ may be found in the appendix of \cite{breuil2002multiplicites}, or in \cite{bushnell2006langlands}. The starting points for the construction are the \emph{hereditary} $\dedekind$-\emph{orders}. For our purposes, we may simply say that the $G$-conjugacy classes of hereditary orders in $\Mat_2(F)$ are represented by the maximal order $\MM=\Mat_2(\dedekind)$, and the Iwahori order $\II$, which consists of those matrices in $\MM$ which are upper-triangular modulo $\frak{p}$. The parahoric subgroups of $G$ are then the groups of units of these rings. Letting $U_\MM=\MM^\times$ and $U_\II=\II^\times$, we may take as representatives for the $\bar{G}$-conjugacy classes of parahoric subgroups of $\bar{G}$ the groups $\bar{U}_\MM=U_\MM\cap\bar{G}$, its conjugate $^\eta\bar{U}_\MM$, and $\bar{U}_\II=U_\II\cap\bar{G}$.\\

We also require the Jacobson radicals of these hereditary orders. The radical of $\MM$ is $\PP_\MM=\Mat_2(\frak{p})$, and the radical of $\II$ is the ideal $\PP_\II$ of matrices which are strictly upper-triangular modulo $\frak{p}$. Given a hereditary order $\AA$, we may then define a filtration of $U_\AA$ by compact open subgroups, by setting $U_\AA^n=1+\PP_\AA^n$, for $n\geq 1$. There is an integer $e_\AA$ called the $\dedekind$-\emph{lattice period} associated to each hereditary order; it is the positive integer $e_\AA$ such that $\PP_\AA^{e_\AA}=\varpi\AA$. The construction of the simple types $(J,\lambda)$ is then by \emph{simple strata}. Roughly speaking, any type $(J,\lambda)$ for a supercuspidal representation $\pi$ of $G$ is constructed via a triple $[\AA,n,\beta]$ consisting of a hereditary $\dedekind$-order $\AA$, the integer $n$ such that $n/e_\AA$ is the depth of $\pi$, and an element $\beta$ of $\PP_\AA^{-n}$ such that $E:=F[\beta]$ is a field. For our purposes, it suffices to know that $J=\dedekind_E^\times U_\AA^{\floor{\frac{n+1}{2}}}$. We will also briefly make use of certain filtration subgroups of $J$: for an integer $k\geq 1$, let $J^k=J\cap U_\AA^k$.\\

These constructions lead, for each supercuspidal representation $\pi$ of $G$, to an irreducible representation $\lambda$ of a compact open subgroup $J$ of $G$, such that $(J,\lambda)$ is a $[G,\pi]_G$-type and there exists a unique extension $\Lambda$ of $\lambda$ to the $G$-normalizer $\tilde{J}$ of $\lambda$ such that $\pi\simeq\cInd_{\tilde{J}}^G\ \Lambda$. Any $\frak{s}$-type arising from these constructions is a \emph{(maximal) G-simple type}. The other main fact that we will require is the ``intertwining implies conjugacy'' property (\cite{bushnell1993admissible}, Theorem 5.7.1), which says that, if we have two maximal simple types $(J,\lambda)$ and $(J',\lambda')$ such that $\III_G(\lambda,\lambda')\neq\emptyset$, then $(J,\lambda)$ and $(J',\lambda')$ must actually be $G$-conjugate.\\

In our case, the simple types in $\bar{G}$ are easily obtained from those in $G$. Let $\pi$ be a supercuspidal representation of $G$, so that $\pi\downharpoonright_{\bar{G}}$ splits into a finite sum of supercuspidal representations of $\bar{G}$. Choose a simple type $(J,\lambda)$ extending to $(\tilde{J},\Lambda)$ such that $\pi\simeq\cInd_{\tilde{J}}^G\ \Lambda$, so that we may perform a Mackey decomposition to obtain
\begin{align*}
\pi\downharpoonright_{\bar{G}}\simeq\bigoplus_{\bar{G}\backslash G/\tilde{J}}\cInd_{^g\bar{J}}^{\bar{G}}\ ^g\bar{\lambda},
\end{align*}
where $\bar{\lambda}=\lambda\downharpoonright_{\bar{J}}$. This is a finite length sum, and the summands will generally be reducible of finite length. However, in our case all ramification is tame and this is actually a decomposition into irreducibles, with one family of exceptions: for the unramified twists of the ``exceptional depth zero'' supercuspidal representation of $G$, which under local Langlands corresponds to the triple imprimitive representation of the Weil group, each of the above summands is reducible of length $2$. We then define the \emph{(maximal)} $\bar{G}$-\emph{simple types} to be the irreducible components of the representations $^g\bar{\lambda}$, for $(J,\lambda)$ running over the $G$-simple types. Given such a $\bar{G}$-simple type $(\bar{J},\mu)$, we have that $\III_{\bar{G}}(\mu)=\bar{J}$; thus they induce up to a supercuspidal representation of $\bar{G}$, and it is clear that this gives a construction of all of the supercuspidals of $\bar{G}$. Just as in the case of $G$, we have an intertwining implies conjugacy property: if two maximal $\bar{G}$-simple types $(\bar{J},\mu)$ and $(\bar{J}',\mu')$ are such that $\III_{\bar{G}}(\mu,\mu')\neq\emptyset$, then there exists a $g\in\bar{G}$ such that $(\bar{J}',\mu')\simeq({}^g\bar{J},{}^g\mu)$ (\cite{bushnell1993admissible}, Theorem 5.3 and Corollary 5.4).

\subsection{The local Langlands correspondence for supercuspidals}

Some of our results on supercuspidals will require a basic understanding of the relevant local Langlands correspondences, which we quickly recall here. Fix once and for all a choice $\bar{F}/F$ of separable algebraic closure. We have a natural projection $\Gal(\bar{F}/F)\twoheadrightarrow\Gal(\bar{\mathfrak{k}}/\mathfrak{k})$, constructed by viewing $\Gal(\bar{F}/F)$ as an inverse limit over finite Galois extensions. The kernel of this map map is the \emph{inertia group} of $F$, which we denote by $I_F$. Let $W_F$ denote the Weil group, which as an abstract group is given by the subgroup of $\Gal(\bar{F}/F)$ generated by $I_F$ and the Frobenius elements, and is then topologized so that $I_F$ is an open subgroup of $W_F$, on which the subspace topology coincides with its topology inherited from $\Gal(\bar{F}/F)$. While in general one requires the full Weil--Deligne group, we will only consider the Langlands correspondence for supercuspidal representations; thus we may simply work with the Weil group.\\

For a $p$-adic group $\Gg$, let $\Irr_{\mathrm{scusp}}(\Gg)$ denote the set of equivalence classes of supercuspidal representations of $\Gg$. Let $\Ll^0(\Gg)$ denote the set of irreducible $L$-parameters for $\Gg$, which is the same as the set of irreducible Frobenius-semisimple representations $W_F\rightarrow\GL2(\C)$, i.e. those irreducible representations under which some fixed Frobenius element of $W_F$ acts semisimply. Then the local Langlands correspondence for $G$ provides a unique natural bijection $\rec:\Irr_{\mathrm{scusp}}(G)\leftrightarrow\Ll^0(G)$, which preserves $L$- and $\varepsilon$-factors, as well as mapping supercuspidal representations to irreducible $L$-parameters, among a list of other properties.\\

From this, as shown in \cite{labesse1979sl2} and \cite{gelbart1982sln}, one may deduce a Langlands correspondence for the supercuspidal representations of $\bar{G}$, which suffices for our purposes. Let $\Ll^0(\bar{G})$ be the image of $\Ll^0(G)$ under the natural map $\Hom(W_F,\GL2(\C))\rightarrow\Hom(W_F,\mathbf{PGL}_2(\C))$. Then we define the local Langlands correspondence $\overline{\rec}$ on $\Irr_{\mathrm{scusp}}(\bar{G})$ by requiring that, for \emph{any} map $R$ which sends a supercuspidal representation $\pi$ of $G$ to one of the irreducible components of $\Res_{\bar{G}}^G\ \pi$, the diagram
\[\xymatrix{
\Irr_{\mathrm{scusp}}(G)\ar@{<->}[r]^-{\rec}\ar[d]_R & \Ll^0(G)\ar@{->>}[d]\\
\Irr_{\mathrm{scusp}}(\bar{G})\ar@{-->}[r]^-{\overline{\rec}} & \Ll^0(\bar{G})
}\]
commutes. The (supercuspidal) $L$-packets are then simply the finite fibres of the map $\overline{\rec}$, which are precisely the sets of irreducible components of the restrictions to $\bar{G}$ of supercuspidal representations of $G$.\\

One may use the local Langlands correspondence to give an alternative description of $G$-inertial equivalence classes of supercuspidal representations: two supercuspidal representations $\pi$ and $\pi'$ of $G$ are inertially equivalent if and only if $\rec(\pi)\downharpoonright_{I_F}\simeq\rec(\pi')\downharpoonright_{I_F'}$.

\section{The main unicity result}\label{unicity}

We now begin working towards the main results, beginning with a description of the relationship between archetypes in $G$ and those in $\bar{G}$.

\begin{lemma}\label{order2}
Let $\pi$ be a supercuspidal representation of $G$, let $\bar{\pi}$ be an irreducible component of $\pi\downharpoonright_{\bar{G}}$, and suppose that $\bar{\pi}$ admits an archetype ${}^{\bar{G}}(\bar{K},\bar{\tau})$. Let $\Psi$ be an irreducible subquotient of $\Ind_{\bar{K}}^K\ \bar{\tau}$ which is contained in $\pi\downharpoonright_K$, and let $\SS=\{[G,\pi\otimes(\chi\circ\det)]_G\ | \ \chi\in\XXX_2(F)\}$. Then $\Psi$ is $\SS$-typical.
\end{lemma}

\begin{proof}
We first note that such a $\Psi$ clearly exists: let $\omega_\pi$ denote the central character of $\pi$, and write $\omega_\pi^0$ for its restriction to $\dedekind^\times$. Let $\tilde{\tau}$ be the extension to $\dedekind^\times\bar{K}$ of $\bar{\tau}$ by $\omega_\pi^0$. Then, by Frobenius reciprocity, some irreducible quotient of $\cInd_{\dedekind^\times\bar{K}}^K\ \tilde{\tau}$ must be contained in $\pi$ upon restriction to $K$. From now on, $\Psi$ will always denote this representation.\\

Let $\pi'$ be an irreducible representation of $G$, and suppose that $\Hom_K(\pi'\downharpoonright_K,\Psi)\neq 0$. Then
\begin{align*}
0&\neq\Hom_K(\Ind_{\bar{K}}^K\ \bar{\tau},\Res_K^G\ \pi')\\
&=\Hom_{\bar{K}}(\bar{\tau},\Res_{\bar{K}}^{\bar{G}}\Res_{\bar{G}}^G\ \pi').
\end{align*}
Since $^{\bar{G}}(\bar{K},\bar{\tau})$ is an archetype for $\bar{\pi}$, we see that $\pi'$ must contain $\bar{\pi}$ upon restriction to $\bar{G}$, so that $\pi'$ is of inertial support $[G,\pi\otimes(\chi\circ\det)]_G$, for some $\chi\in\XXX(F)$. Comparing central characters, $\chi$ must be an unramified twist of a quadratic character, as required.
\end{proof}

\begin{lemma}\label{extensiontok}
The representation $\Psi$ constructed in Lemma \ref{order2} is a $[G,\pi]_G$-type.
\end{lemma}

\begin{proof}
By Lemma \ref{order2}, it remains only to rule out the possibility that $\Psi$ is contained in a representation of inertial support $[G,\pi\otimes(\chi\circ\det)]_G$, for some non-trivial $\chi\in\XXX_2(F)$. Indeed, this would show that $\Psi$ is $[G,\pi]_G$-typical, and the unicity of types for $G$ would immediately imply that $\Psi$ represents a $[G,\pi]_G$-archetype. We now argue by cases.\\

As $\pi$ is a supercuspidal representation, we may write $\pi\simeq\cInd_{\tilde{J}}^G\ \Lambda$, with $(\tilde{J},\Lambda)$ extending a maximal simple type $(J,\lambda)$ contained in $\pi$. Suppose for contradiction that $\Psi$ is not a type. As noted by Henniart in the appendix to \cite{breuil2002multiplicites}, paragraphs A.2.4 -- A.2.7, A.3.6 -- A.3.7 and A.3.9 -- A.3.11, every irreducible component of $\pi\downharpoonright_K$ other than $\tau$ appears in the restriction to $K$ of either a parabolically induced representation, or some other supercuspidal $\pi_\mu$, in a different inertial equivalence class to that of $\pi$, which we may now describe explicitly. There are three further subcases which we treat separately.\\

Suppose first that $\Psi$ is contained in some parabolically induced representation. We may therefore find a character $\zeta$ of $T$ such that $\Psi$ is isomorphic to some irreducible component of $\Res_K^G\Ind_{T,B}^G\ \zeta$. As $\Ind_{\bar{K}}^K\ \bar{\tau}$ projects onto $\Psi$, we therefore have
\begin{align*}
0&\neq\Hom_K(\cInd_{\bar{K}}^K\ \bar{\tau},\Res_K^G\Ind_{T,B}^G\ \zeta)\\
&=\Hom_G(\cInd_{\bar{K}}^{\bar{G}}\ \bar{\tau},\Res_{\bar{G}}^G\Ind_{T,B}^G\ \zeta)\\
&=\bigoplus_{i=1}^n\Hom_G(\bar{\pi},\Ind_{\bar{T},\bar{B}}^{\bar{G}}\Res_{\bar{T}}^T\ \zeta).
\end{align*}
Here, $n$ is the integer such that $\cInd_{\bar{K}}^{\bar{G}}\ \bar{\tau}\simeq\bar{\pi}^{\oplus n}$, which exists by Proposition 5.2 of \cite{bushnell1998structure}, and the final equality follows from a Mackey decomposition with the summation involved being trivial as $B\bar{G}=G$. Hence $\bar{\pi}$ is contained in some parabolically induced representation, which provides a contradiction by Lemma \ref{order2}.\\

Now suppose $\Psi$ does not appear as an irreducible component of the restriction to $K$ of any parabolically induced representation, and suppose furthermore that $\pi$ is of integral depth $n$. In this case, we may construct a new supercuspidal representation containing every irreducible component of $\pi\downharpoonright_K$ other than the archetype $\tau$. Let $E/F$ be the unique unramified quadratic extension of $F$, and choose an embedding $\dedekind_E^\times\subset K$. Let $\mu$ be any level $1$ character of $E^\times$ trivial on $F^\times$, and let $(J,\lambda)$ be a simple type for $\pi$. Then the pair $(J,\lambda\otimes\mu)$ is again a maximal simple type contained in some supercuspidal representation $\pi_\mu$ lying in a different inertial equivalence class to that of $\pi$, and any irreducible component of $\pi\downharpoonright_K$ other than $\tau$ must be contained in $\pi_\mu$ upon restriction; in particular, we must have $\sigma\hookrightarrow\pi_\mu\downharpoonright_K$. But then $\pi_\mu$ must be isomorphic to an unramified twist of $\pi\otimes(\chi\circ\det)$, for some (non-trivial by assumption) $\chi\in\XXX_2(F)$, which is to say that their archetypes must coincide. The archetype for $\pi_\mu$ is $\cInd_J^K\ \lambda\otimes\mu$, and the archetype for $\pi\otimes(\chi\circ\det)$ is $(\cInd_J^K\ \lambda)\otimes(\chi\circ\det)$. If these two representations are isomorphic, then we must have $\lambda\otimes\mu\simeq\lambda\otimes(\chi\circ\det)$, as $\III_K(\lambda\otimes\mu,\lambda\otimes(\chi\circ\det))\neq\emptyset$, and if $g$ intertwines $\lambda\otimes\mu$ with $\lambda\otimes(\chi\circ\det)$, then $g$ intertwines $\lambda\downharpoonright_{J^1}$ with itself, and so $g\in J$. As $\lambda\otimes\mu\simeq\lambda\otimes(\chi\circ\det)$, we may use Schur's lemma to obtain $0\neq\Hom_J(\lambda\otimes\mu,\lambda\otimes(\chi\circ\det))\subseteq\End_{J^1}(\lambda\downharpoonright_{J^1})=\C$, and hence $\Hom_J(\lambda\otimes\mu,\lambda\otimes(\chi\circ\det))$ contains the identity map, so that we must have $\mu=\chi\circ\det$ on $\dedekind_E^\times$. However, there are only two quadratic characters $\chi$ of $F^\times$ while there are $q+1\geq 4$ such characters $\mu$. Choosing $\mu$ non-quadratic, we obtain a contradiction.\\

Finally, consider the case where $\pi$ is of half-integral depth and $\Psi$ is not contained in any parabolically induced representation, we argue essentially as before. We assume further that $\pi$ is of depth at least $\frac{3}{2}$; for $\pi$ of depth $\frac{1}{2}$ Henniart shows that this case never arises. Let $E/F$ be the ramified quadratic extension associated to the simple type for $\pi$, and choose an embedding $\dedekind_E^\times\subset U_\II\subset K$. For $\mu$, we take a level $2$ character of $E^\times$ trivial on $\dedekind_F^\times$, and construct $\pi_\mu$ as before. Letting $(J,\lambda)$ be a simple type for $\pi$, so that $\lambda$ is one-dimensional, the pair $(J,\lambda\otimes\mu)$ is a simple type for some supercuspidal $\pi_\mu$ in a different inertial equivalence class to that of $\pi$, and $\Psi$ must appear in the restriction to $K$ of $\pi_\mu$. Then, up to an unramified twist, $\pi_\mu\simeq\pi\otimes(\chi\circ\det)$ for $\chi$ a nontrivial quadratic character of $F^\times$, and so the archetypes $\cInd_J^K\ \lambda\otimes\mu$ and $(\cInd_J^K\ \lambda)\otimes(\chi\circ\det)$ coincide. Then $\III_K(\lambda\otimes\mu,\lambda\otimes(\chi\circ\det))\neq\emptyset$, and if $g\in\III_K(\lambda\otimes\mu,\lambda\otimes(\chi\circ\det))$, then $g\in\III_K(\lambda\downharpoonright_{J^2},\lambda\downharpoonright_{J^2})=J$, as $\pi$ is of depth at least $\frac{3}{2}$, and $\lambda$ is one-dimensional so that $\lambda\downharpoonright_{J^2}$ is a simple character. It follows that we must have $\lambda\otimes\mu\simeq\lambda\otimes(\chi\circ\det)$, and so $\mu\simeq\chi\circ\det$. But then $\mu$ is of level $2$ while $\chi\circ\det$ is tame and hence of level at most $1$, giving a contradiction and completing the proof.
\end{proof}

\begin{lemma}\label{lemma}
Let $\pi$ be a supercuspidal representation of $G$, and let $^G(K,\tau)$ be an archetype for $\pi$. Then every irreducible component of $\tau\downharpoonright_{\bar{K}}$ is induced from a Bushnell--Kutzko type for $\bar{G}$.
\end{lemma}

\begin{proof}
By Theorem \ref{unicitygl2}, the representation $\tau$ is of the form $\tau=\cInd_J^K\ \lambda$, where $(J,\lambda)$ is a maximal simple type contained in $\pi$. Then we may perform a Mackey decomposition to obtain
\begin{align*}
\Res_{\bar{K}}^K\cInd_J^K\ \lambda=\bigoplus_{J\bar{K}\backslash K}\cInd_{^g\bar{J}}^{\bar{K}}\ ^g\bar{\lambda},
\end{align*}
where $\bar{\lambda}=\lambda\downharpoonright_{\bar{J}}$. The irreducible components of this representation are all of the required form.
\end{proof}

\begin{theorem}\label{typesaresimple}
Let $\bar{\pi}$ be a supercuspidal representation of $\bar{G}$.
\begin{enumerate}
\item If $^{\bar{G}}(\mathscr{K},\bar{\tau})$ is an archetype for $\bar{\pi}$, then there exists a simple type $(\bar{J},\mu)$ with $\bar{J}\subset\mathscr{K}$ such that $\bar{\tau}\simeq\cInd_{\bar{J}}^{\mathscr{K}}\ \mu$.
\item If $(\bar{J},\mu)$ is a simple type contained in $\bar{\pi}$ and $\mathscr{K}$ is a maximal compact subgroup of $\bar{G}$ which contains $\bar{J}$, then the representation $\cInd_{\bar{J}}^{\mathscr{K}}\ \mu$ is the unique $[\bar{G},\bar{\pi}]_{\bar{G}}$-typical representation of $\mathscr{K}$.
\end{enumerate}
\end{theorem}

\begin{proof}
For \emph{(i)} we may reduce to the case $\mathscr{K}=\bar{K}$. By Lemma \ref{extensiontok}, $\bar{\tau}$ is an irreducible component of the restriction to $\bar{K}$ of the unique archetype $(K,\tau)$ for some irreducible representation $\pi$ of $G$ containing $\bar{\pi}$ upon restriction to $\bar{G}$. As there exists a simple type $(J,\lambda)$ such that $\tau=\cInd_J^K\ \lambda$, the result follows immediately from Lemma \ref{lemma}.\\

To see \emph{(ii)}, it remains to check that, given two distinct simple types $(\bar{J},\mu)$ and $(\bar{J}',\mu')$ contained in $\bar{\pi}$ which are, moreover, contained in the same conjugacy class of maximal compact subgroups, these simple types provide the same archetypes through induction. Thus, we may as well assume that $\bar{J},\bar{J}'\subset\bar{K}$. As $(\bar{J},\mu)$ and $(\bar{J}',\mu')$ are $\frak{s}$-types, $\pi$ will appear as a subquotient of the induced representations $\cInd_{\bar{J}}^{\bar{G}}\ \mu$ and $\cInd_{\bar{J}'}^{\bar{G}}\ \mu'$; hence we will have
\begin{align*}
0&\neq\Hom_{\bar{G}}(\cInd_{\bar{J}}^{\bar{G}}\ \mu,\cInd_{\bar{J}'}^{\bar{G}}\ \mu')\\
&=\Hom_{\bar{J}}(\mu,\Res_{\bar{J}}^{\bar{G}}\cInd_{\bar{J}'}^{\bar{G}}\ \mu')\\
&=\bigoplus_{\bar{J}'\backslash\bar{G}/\bar{J}}\Hom_{\bar{J}}(\mu,\cInd_{^g\bar{J}\cap\bar{J}'}^{\bar{J}}\Res_{^g\bar{J}\cap\bar{J}'}^{^g\bar{J}}\ ^g\mu')\\
&=\bigoplus_{\bar{J}'\backslash\bar{G}/\bar{J}}\Hom_{^g\bar{J}'\cap\bar{J}}(\Res_{^g\bar{J}'\cap\bar{J}}^{\bar{J}}\ \mu,\Res_{^g\bar{J}'\cap\bar{J}}^{^g\bar{J}'}\ ^g\mu'),
\end{align*}
and so $\III_{\bar{G}}(\mu,\mu')\neq\emptyset$. As $\pi$ is supercuspidal then $(\bar{J},\mu)$ and $(\bar{J}',\mu')$ will be simple types, and so by the intertwining implies conjugacy property there will exist a $g\in\bar{G}$ such that $^g(\cInd_{\bar{J}}^{\bar{K}}\ \mu)\simeq\cInd_{\bar{J}'}^{^g\bar{K}}\ \mu'$. As $\bar{J}'$ is contained in at most one maximal compact subgroup in each $\bar{G}$-conjugacy class, we must actually have $^g\bar{K}'=\bar{K}$, and so $(\bar{J},\mu)$ and $(\bar{J}',\mu')$ induce to the same archetype.
\end{proof}

\section{An explicit description of supercuspidal archetypes and a result on $L$-packets}\label{supercuspidals}

Having completed the proof of Theorem \ref{typesaresimple}, we now provide a more explicit description of the theory of archetypes for supercuspidal representations. Given a supercuspidal representation $\bar{\pi}$ of $\bar{G}$, we define the \emph{ramification degree} $e_{\bar{\pi}}$ of $\bar{\pi}$ to be $1$ if $\bar{\pi}$ is of integral depth, or $2$ if $\bar{\pi}$ is of half-integral depth. Then we obtain the following corollary to Theorem \ref{typesaresimple}:

\begin{corollary}
Let $\bar{\pi}$ be a supercuspidal representation of $\bar{G}$. Then the number of $[\bar{G},\bar{\pi}]_{\bar{G}}$-archetypes is precisely $e_{\bar{\pi}}$.
\end{corollary}

\begin{proof}
It remains only for us to count the number of archetypes obtained by inducing a maximal simple type contained in $\bar{\pi}$ up to maximal compact subgroups. If $\bar{\pi}$ is ramified then, up to conjugacy, any simple type for $\bar{\pi}$ is defined on a group contained in the Iwahori subgroup $\bar{U}_\II$ of $\bar{G}$, which is itself contained in both $\bar{K}$ and $^\eta\bar{K}$; hence ramified supercuspidals admit two archetypes. If $\bar{\pi}$ is unramified, it suffices to show that the subgroup $\bar{J}$ on which any simple type $\mu$ for $\bar{\pi}$ is defined embeds into precisely one $\bar{G}$-conjugacy class of maximal compact subgroups. Without loss of generality, we may as well assume that $\bar{J}\subseteq\bar{U}_\MM$. We have $\ker\NNN_{E/F}\subseteq\bar{J}\subseteq\bar{U}_\MM$, where $\NNN_{E/F}$ is the norm map on the (unramified) quadratic extension $E/F$ associated to $\bar{\pi}$. Suppose for contradiction that we also have $\bar{J}\subseteq{}^\eta\bar{U}_\MM$. As the group $\ker\NNN_{E/F}$ contains the group $\mu_{q+1}$ of $(q+1)$-th roots of unity, we would therefore also have $\mu_{q+1}\subset\bar{U}_\MM\cap{}^\eta\bar{U}_\MM=\bar{U}_\II$. However, the Iwahori subgroup contains no order $q+1$ elements, giving the desired contradiction.\\

Thus, the only way in which one might obtain two archetypes when $e_{\bar{\pi}}=1$ is if $\bar{\pi}$ contains simple types which are $G$-conjugate but not $\bar{G}$-conjugate; this clearly cannot be the case by the intertwining implies conjugacy property.
\end{proof}

This completely describes the number of archetypes contained in any supercuspidal representation of $\bar{G}$. We now prove a complementary result, which allows us to describe the relationship between the theories of archetypes for $\bar{G}$ and $G$. We first require a converse result to Lemma \ref{extensiontok}.

\begin{lemma}\label{restrictiontokbar}
Let $\pi$ be a supercuspidal representation of $G$, let $\mathfrak{s}=[G,\pi]_G$, and let $^G(K,\tau)$ be the unique $\mathfrak{s}$-archetype. Let $\bar{\pi}$ be an irreducible component of $\pi\downharpoonright_{\bar{G}}$. Then there exists a $g\in G$ and an irreducible component $\bar{\tau}$ of $^g\tau\downharpoonright_{^g\bar{K}}$ such that $^{\bar{G}}({}^g\bar{K},\bar{\tau})$ is an archetype for $\bar{\pi}$.
\end{lemma}

\begin{proof}
We may assume without loss of generality, by conjugating by $\eta$ if necessary, that $\bar{\pi}=\cInd_{\bar{K}}^{\bar{G}}\ \rho$, where $\rho=\cInd_{\bar{J}}^{\bar{K}}\ \mu$ is the induction to $\bar{K}$ of a $\bar{G}$-simple type. Let $\{\bar{\tau}_j\}$ be the (finite) set of irreducible components of $\tau\downharpoonright_{\bar{K}}$. We first show that any $\pi'\in\Irr(\bar{G})$ containing one of the $\bar{\tau}_j$ upon restriction must appear in the restriction to $\bar{G}$ of $\pi$. We have
\begin{align*}
0&\neq\bigoplus_j\Hom_{\bar{K}}(\bar{\tau}_j,\pi')\\
&=\Hom_{\bar{K}}(\Res_{\bar{K}}^K\ \tau,\Res_{\bar{K}}^{\bar{G}}\ \pi')\\
&=\Hom_{\bar{G}}(\cInd_{\bar{K}}^{\bar{G}}\Res_{\bar{K}}^K\ \tau,\pi'),
\end{align*}
and so we obtain $\pi'\twoheadleftarrow\cInd_{\bar{K}}^{\bar{G}}\Res_{\bar{K}}^K\ \tau\hookrightarrow\Res_{\bar{G}}^G\cInd_K^G\ \tau$. Every irreducible subquotient of the representation $\cInd_K^G\ \tau$ is a twist of $\pi$, and hence coincides with $\pi$ upon restriction to $\bar{G}$, so that any such representation $\pi'$ must be a subrepresentation of the restriction to $\bar{G}$ of $\pi$. Hence the possible representations $\pi'$ all lie in a single $G$-conjugacy class of representations of $\bar{G}$. Let $g\in G$ be such that $^g\pi'\simeq\bar{\pi}$, so that $\pi'\simeq\cInd_{^{g}\bar{K}}^{\bar{G}}\ ^g\rho$, and choose $j$ so that $\pi'$ contains $\bar{\tau}_j$. We claim that $(^g\bar{K},{}^g\bar{\tau}_j)$ is the required type.\\

It suffices to show that any $G$-conjugate of $\bar{\pi}$ containing $({}^g\bar{K},{}^g\bar{\tau}_j)$ is isomorphic to $\bar{\pi}$. Suppose that, for some $h\in G$, we have $\Hom_{^g\bar{K}}({}^h\bar{\pi},{}^g\bar{\tau}_j)\neq 0$. The representation $^h\bar{\pi}$ is of the form $^h\bar{\pi}=\cInd_{^h\bar{J}}^{\bar{G}}\ ^h\mu$ and, using Lemma \ref{lemma}, we see that the representation $^g\bar{\tau}_j$ must be induced from some $\bar{G}$-simple type $(\bar{J}',\mu')$, say. Then
\begin{align*}
0&\neq\Hom_{^g\bar{K}}(\Res_{^g\bar{K}}^{\bar{G}}\ \bar{\pi},{}^g\bar{\tau}_j)\\
&=\Hom_{\bar{J}'}(\Res_{\bar{J}'}^{\bar{G}}\cInd_{^h\bar{J}}^{\bar{G}}\ ^h\mu,\mu')\\
&=\bigoplus_{^h\bar{J}\backslash\bar{G}/\bar{J}'}\Hom_{\bar{J}'}(\cInd_{^{xh}\bar{J}\cap\bar{J}'}^{\bar{J}'}\Res_{^{xh}\bar{J}\cap\bar{J}'}^{^{xh}\bar{J}}\ ^{xh}\mu,\mu')\\
&=\bigoplus_{^h\bar{J}\backslash\bar{G}/\bar{J}'}\Hom_{^{xh}\bar{J}\cap\bar{J}'}(\Res_{^{xh}\bar{J}\cap\bar{J}'}^{^{xh}\bar{J}}\ ^{xh}\mu,\Res_{^{xh}\bar{J}\cap\bar{J}'}^{\bar{J}'}\ \mu').
\end{align*}
Then $^h\mu$ and $\mu'$ must intertwine in $\bar{G}$, and the intertwining implies conjugacy property shows that the types $^h\mu$ and $\mu'$ must actually be $\bar{G}$-conjugate, and hence $^h\bar{\pi}$ is $\bar{G}$-conjugate to $^g\pi'\simeq\bar{\pi}$. Therefore $^h\bar{\pi}\simeq\bar{\pi}$, and the result follows.
\end{proof}

We are then able to give a description of the relationship between the archetypes in the two groups $G$ and $\bar{G}$ in terms of $L$-packets.

\begin{proposition}\label{lpackets}
Let $\pi$ be a supercuspidal representation of $G$, let $\mathfrak{s}=[G,\pi]_G$, and let $^G(K,\tau)$ be the unique $\mathfrak{s}$-archetype. Let $\Pi$ be the $L$-packet of irreducible components of $\pi\downharpoonright_{\bar{G}}$. Then the set of archetypes for the representations in $\Pi$ is precisely the set of the $^{\bar{G}}(\mathscr{K},\bar{\tau})$, for $(\mathscr{K},\bar{\tau})$ an irreducible component of either $\tau\downharpoonright_{\bar{K}}$ or $^\eta\tau\downharpoonright_{^\eta\bar{K}}$.
\end{proposition}

\begin{proof}
We show that the set of typical representations of $\bar{K}$ for some $\bar{\pi}\in\Pi$ is equal to the set of irreducible components of $\tau\downharpoonright_{\bar{K}}$; the general result then follows immediately. Let $(\bar{K},\bar{\tau})$ be an archetype for some $\bar{\pi}\in\Pi$. Applying Lemma \ref{extensiontok}, $\bar{\tau}$ is of the required form. Conversely, the irreducible components of $\tau\downharpoonright_{\bar{K}}$ are all $K$-conjugate by Clifford theory, and so if one of them is a type for some element of $\Pi$ then they all must be. Applying Lemma \ref{restrictiontokbar}, at least one of these irreducible components must be a type for some $\bar{\pi}\in\Pi$.
\end{proof}

\begin{corollary}[Inertial local Langlands correspondence for $\SL2(F)$]\label{inertialcorrespondence}
Let $\varphi:I_F\rightarrow\PGL_2(\C)$ be a representation extending to an irreducible $L$-parameter $\tilde{\varphi}:W_F\rightarrow\PGL_2(\C)$. Then there exists a finite set $\{(\mathscr{K}_i,\tau_i)\}$ of smooth irreducible representations $\tau_i$ of maximal compact subgroups $\mathscr{K}_i$ of $\bar{G}$ such that, for all smooth, irreducible, infinite-dimensional representations $\pi$ of $\bar{G}$, we have that $\pi$ contains some $\tau_i$ upon restriction to $\mathscr{K}_i$ if and only if $\overline{\rec}(\pi)\downharpoonright_{I_F}\simeq\varphi$. Furthermore, this set is unique up to $\bar{G}$-conjugacy.
\end{corollary}

\begin{proof}
Let $\Pi=\overline{\rec}^{-1}(\tilde{\varphi})$ be the $L$-packet corresponding to $\tilde{\varphi}$, so that $\Pi$ is the set of irreducible components upon restriction to $\bar{G}$ of some supercuspidal representation $\sigma$ of $G$. Let $\psi=\rec(\sigma)$, so that, by Corollary 8.2 of \cite{paskunas2005unicity}, there exists a unique smooth irreducible representation $\tau$ of $K$ such that, for all smooth, irreducible, infinite-dimensional representations $\rho$ of $G$, we have that $\rho$ contains $\tau$ upon restriction to $K$ if and only if $\rec(\rho)\downharpoonright_{I_F'}\simeq\psi\downharpoonright_{I_F'}$. Then $^G(K,\tau)$ is the unique archetype for $\sigma$, and the set $\{{}^{\bar{G}}(\mathscr{K}_i,\tau_i)\}$ of archetypes for $\Pi$ is precisely that represented by the finite set of irreducible components of $\tau\downharpoonright_{\bar{K}}$ and $^\eta\tau\downharpoonright_{^\eta\bar{K}}$. Let $\SS$ be the set of $\bar{G}$-inertial equivalence classes of representations in $\Pi$. Then, as each of the $(\mathscr{K}_i,\tau_i)$ is an archetype, it follows that, for all smooth, irreducible, infinite-dimensional representations $\pi$ of $\bar{G}$, we have that $\pi$ contains one of the $\tau_i$ upon restriction to $\mathscr{K}_i$ if and only if $[\bar{G},\pi]_{\bar{G}}\in\SS$, if and only if $\pi\in\Pi$, if and only if $\overline{\rec}(\pi)\downharpoonright_{I_F}\simeq\varphi$, as required.
\end{proof}

\bibliographystyle{amsalpha}
\addcontentsline{toc}{section}{References}
\bibliography{Archetypes}

\end{document}